\documentclass{amsart}
\usepackage{graphicx}
\usepackage{verbatim}
\usepackage{textcomp}
\newtheorem{thm}{Theorem}[section]
\newtheorem{cor}[thm]{Corollary}
\newtheorem{lem}[thm]{Lemma}
\newtheorem{prop}[thm]{Proposition}
\theoremstyle{definition}
\newtheorem{defn}[thm]{Definition}
\theoremstyle{remark}
\newtheorem{rem}[thm]{Remark}
\numberwithin{equation}{section}
\newcommand{\norm}[1]{\left\Vert#1\right\Vert}

\newcommand{\Real}{\mathbb R}

\newcommand{\To}{\longrightarrow}

\newcommand{\dist}[0]{\mathrm{dist}}

\title{Existence of Good Sweepouts on Closed Manifolds}
\author{Longzhi Lin}%
\address{Department of Mathematics\\
Johns Hopkins University\\
3400 N. Charles St.\\
Baltimore, MD 21218}
\author{Lu Wang}%
\address{Department of Mathematics\\
Massachusetts Institute of Technology\\
77 Massachusetts Avenue, Cambridge, MA 02139}

\email{lzlin@math.jhu.edu and luwang@math.mit.edu}

\begin{document}
\maketitle
\begin{abstract}
 In this note we establish estimates for the harmonic map heat flow from $S^1$ into a closed manifold, and use it to construct sweepouts with the following good property: each curve in the tightened sweepout, whose energy is close to the maximal energy of curves in the sweepout, is itself close to a closed geodesic.
\end{abstract}

\section{Introduction}
Given a minimizing sequence of sweepouts of the width (see (\ref{Fourtheqn})), we apply the harmonic map heat flow on each curve in the sweepout to pull it tight while preserving the sweepout. Moreover the tightened sweepout has the following good property (see Theorem \ref{SecondMainThm}): each curve in the tightened sweepout whose energy is close to the maximal energy of curves in the sweepout is itself close to a closed geodesic. In particular, the width is the energy of some closed geodesic.
On closed non-simply-connected surfaces, Grayson showed that there exist simple closed geodesics in each nontrivial $\pi_1$ homotopy class by the curve shortening flow; see \cite{G}. On the $2$-sphere whose $\pi_1$ homotopy group is trivial, Birkhoff used sweepouts to find non-trivial closed geodesics; see \cite{B1}, \cite{B2}, \cite{CM1}, \cite{CM3}, \cite{CD}, \cite{Lin} and section $2$ in \cite{Cr} about Birkhoff's idea. The argument works equally well on other Riemannian manifolds. The Birkhoff's curve shortening process is a discrete gradient flow of the length functional and the key point is a convexity formula for the energy functional which controls the distance of curves from closed geodesics explicitly; see Lemma $4.2$ in \cite{CM1} and cf. Theorem $3.1$ in \cite{CM2}. However, it requires some work to show the discrete shortening process preserves the homotopy class of sweepouts. Instead, we use a continuous method, i.e. the harmonic map heat flow, to tighten sweepouts, which provides a natural homotopy of sweepouts. There are several applications of the existence of good sweepouts on closed manifolds. For instance, in \cite{CM1}, Colding and Minicozzi bounded from above, by a negative constant, the rate of change of the width for a one-parameter family of convex hypersurfaces that flows by mean curvature. The estimate is sharp and leads to a sharp estimate for the extinction time. And a similar bound for the rate of change for the two dimensional width is shown for homotopy $3$-spheres evolving by the Ricci flow; see \cite{CM2} and \cite{P}. \\

$\mathit{Acknowledgements.}$
We would like to thank our advisors Prof. Tobias Colding and Prof. William Minicozzi for their continuous guidance.

\section{Harmonic Map Heat Flow}
Throughout we use the subscripts $\theta$ and $t$ to denote taking partial derivatives of maps with respect to $\theta$ and $t$; $u$ satisfies the harmonic map heat flow equation, which is defined in (\ref{Firsteqn}).  Let $(M,g)$ be a closed Riemannian manifold. By the Nash Embedding Theorem, $M$ can be isometrically embedded into Euclidean space $(\Real^N,\langle,\rangle)$. Given a closed curve $\gamma\in H^1(S^1,M)$, define the energy functional $E(\gamma)=\frac{1}{2}\int_{S^1}|\gamma_\theta|^2d\theta$.
The harmonic map heat flow is the negative $L^2$ gradient flow of the energy functional. Thus the equation of the harmonic map heat flow from $S^1$ into $M$ is 
\begin{equation}
\label{Firsteqn}
u_t=u_{\theta\theta}-A_u(u_\theta,u_\theta)\ \ \text{on}\  (0,\infty)\times S^1;\ \ 
\lim_{t\To 0+}u(t,\cdot)=u_0\ \ \text{in}\  H^1(S^1,M),
\end{equation}
where $A_u$ is the second fundamental form of $M$ in $\Real^N$ at point $u(\theta)$. 
We study the long time existence and uniqueness of the solution of (\ref{Firsteqn}). For relevant results of the two dimensional harmonic map heat flow, the reader could refer to \cite{S1}. Also we would like to thank Tobias Lamm for bringing to our attention the paper \cite{OT} of Ottarsson, which has some overlap with our paper and in which Theorem \ref{FirstMainThm} was proved under the stronger assumption of $C^1$ initial data (and thus the $C^1$ continuity at $t=0$)\footnote[1]{In our setting, the $C^1$ continuity at $t=0$ may not be true. For our purpose that the harmonic map heat flow preserves the homotopy class of sweepouts, we use a different argument to show the $H^1$ continuity at $t=0$.}.

\begin{thm}
\label{FirstMainThm} Given $u_0\in H^1(S^1,M)$, there exists a unique solution $u(t,\theta)\in C^\infty((0,\infty)\times S^1,M)$ of (\ref{Firsteqn}). 
\end{thm}
The following is devoted to the proof of Theorem \ref{FirstMainThm}. First, by the corollary on page 124 of \cite{Ha}, given any initial data $u_0\in C^\infty(S^1,M)$, there exists $T_0>0$ and a unique solution $u\in C^\infty([0,T_0)\times S^1,M)$ of (\ref{Firsteqn}). We show that the solution $u$ can be extended smoothly beyond $T_0$. First, note that the energy is non-increasing under the harmonic map heat flow:
\begin{lem}
\label{SecondLem}For $0\leq t_1\leq t_2<T_0$,
\begin{equation}
\label{Secondeqn}
E(u(t_1,\cdot))-E(u(t_2,\cdot))
=\int_{t_1}^{t_2}\int_{S^1}|u_t|^2d\theta dt.
\end{equation}
\end{lem}
\begin{proof}
Multiply the harmonic map heat equation by $u_t$ and integrate over $[t_1,t_2]\times S^1$,
\begin{eqnarray*}
\int_{t_1}^{t_2}\int_{S^1}|u_t|^2d\theta dt&=&\int_{t_1}^{t_2}\int_{S^1}\langle u_{\theta\theta},u_t\rangle d\theta dt=-\int_{t_1}^{t_2}\int_{S^1}\langle u_\theta,u_{\theta t}\rangle d\theta dt\\
&=&E(u(t_1,\cdot))-E(u(t_2,\cdot)).
\end{eqnarray*}
\end{proof}
Next we derive the gradient bound of $u$.
\begin{lem}
\label{FirstLem} $(\partial_t-\partial_\theta^2)|u_\theta|^2\leq 0$.
\end{lem}
\begin{proof}
\begin{eqnarray*}
\partial_t|u_\theta|^2&=&2\langle u_\theta,u_{\theta t}\rangle=2\langle u_\theta,u_{\theta\theta\theta}\rangle
-2\langle u_\theta,(A_u(u_\theta,u_\theta))_\theta\rangle\\
&=&2\langle u_\theta,u_{\theta\theta\theta}\rangle
+2\langle u_{\theta\theta},A_u(u_\theta,u_\theta)\rangle\\
&=&\partial_\theta^2|u_\theta|^2-2|u_{\theta\theta}|^2
+2|A_u(u_\theta,u_\theta)|^2\\
&\leq&\partial_\theta^2|u_\theta|^2.
\end{eqnarray*}
\end{proof}
Thus, by Theorem 7.36 in \cite{Lie}, for any $\tau>0$ and $(t,\theta)\in [\tau,T_0)\times S^1$,
\begin{equation}
\label{Fiftheqn}
|u_\theta|^2(t,\theta)\leq \tau^{-1}E(u_0).
\end{equation}
By Proposition 7.18 in \cite{Lie}, $|u_{\theta\theta}|$ and $|u_t|$ are bounded on $[2\tau,T_0)\times S^1$. And by induction, for any $\tau>0$, the higher order derivatives of $u$ on $[2\tau,T_0)\times S^1$ are bounded uniformly by constants depending only on $M$, $E(u_0)$, $\tau$ and $T_0$. Hence $u$ can be extended smoothly to a solution of (\ref{Firsteqn}) beyond $T_0$. In other word, there exists a unique solution $u\in C^\infty([0,\infty)\times S^1,M)$ of (\ref{Firsteqn}), if $u_0\in C^\infty(S^1,M)$.

Next, in general, given $u_0\in H^1(S^1,M)$, we can find a sequence $u^m_0\in C^\infty(S^1,M)$ approaching $u_0$ in the $H^1$ topology. Let $u^m$ be the solution of the harmonic map heat flow with initial data $u^m_0$. Thus by (\ref{Fiftheqn}) and discussion above, for any $\tau>0$ and $T_0>\tau$, $u^m$ and all their derivatives are bounded uniformly, independent of $m$. Hence by the Arzela-Ascoli Theorem and a diagonalization argument, there exists a map $u\in C^\infty((0,\infty)\times S^1,M)$ solving the harmonic map heat flow with $E(u(t,\cdot))\leq E(u_0)$. And it follows from the lemma below that $t\To u(t,\cdot)$ is a continuous map from $[0,\infty)\To H^1(S^1,M)$. 
\begin{lem}
\label{FourthLem}
Given $\epsilon>0$, there exists $\delta>0$, depending on $M$, $u_0$ and $\epsilon$, so that if  $0\leq t_1< t_2$ and $t_2-t_1<\delta$, then 
\begin{equation}
\label{Sixtheqn}
\norm{u(t_2,\cdot)-u(t_1,\cdot)}_{H^1(S^1)}\leq \epsilon.
\end{equation}
\end{lem}
\begin{proof}
Note that by Lemma \ref{SecondLem},  $\lim_{t\rightarrow 0}u(t,\cdot)= u_0$ in the $L^2(S^1,M)$ topology. Moreover, we have
\begin{equation*}
\int_{S^1}|u(t_2,\theta)-u(t_1,\theta)|^2d\theta
\leq\int_{S^1}\left|\int_{t_1}^{t_2}u_tdt\right|^2d\theta
\leq(t_2-t_1)\int_{t_1}^{t_2}\int_{S^1}|u_t|^2d\theta dt.
\end{equation*}
Next, by Lemma \ref{SecondLem} and the Cauchy-Schwarz inequality,
\begin{eqnarray*}
&&\int_{S^1}|u_\theta(t_2,\theta)-u_\theta(t_1,\theta)|^2
d\theta\\
&=&\int_{S^1}|u_\theta(t_1,\theta)|^2d\theta
-\int_{S^1}|u_\theta(t_2,\theta)|^2d\theta
-2\int_{S^1}\langle u_\theta(t_2,\theta),
u_\theta(t_1,\theta)-u_\theta(t_2,\theta)\rangle d\theta\\
&=&2\int_{t_1}^{t_2}\int_{S^1}|u_t|^2d\theta dt
+2\int_{S^1}\langle u_{\theta\theta}(t_2,\theta),
u(t_1,\theta)-u(t_2,\theta)\rangle d\theta\\
&\leq&2\int_{t_1}^{t_2}\int_{S^1}|u_t|^2d\theta dt
+2\left(\int_{S^1}|u_{\theta\theta}(t_2,\theta)|^2d\theta\right)
^{\frac{1}{2}}\left(\int_{S^1}
|u(t_2,\theta)-u(t_1,\theta)|^2d\theta\right)^{\frac{1}{2}}\\
&\leq&2\int_{t_1}^{t_2}\int_{S^1}|u_t|^2d\theta dt
+2(t_2-t_1)^{\frac{1}{2}}\left(\int_{S^1}|u_{\theta\theta}(t_2,\theta)|^2d\theta\right)
^{\frac{1}{2}}\left( \int_{t_1}^{t_2}\int_{S^1}|u_t|^2d\theta dt\right)^{\frac{1}{2}}.
\end{eqnarray*}
From Lemma \ref{SecondLem}, (\ref{Firsteqn}) and (\ref{Fiftheqn}),
\begin{equation}
\label{Seventheqn}
\int_{S^1}|u_{\theta\theta}(t_2,\theta)|^2d\theta\leq
\int_{S^1}|u_t(t_2,\theta)|^2d\theta+t_2^{-1}\sup_M|A|^2\cdot
E(u_0)^2.
\end{equation}
We derive the evolution equation of $|u_t|^2$:
\begin{eqnarray*}
\partial_t|u_t|^2&=&2\langle u_t,u_{tt}\rangle
=2\langle u_t,u_{\theta\theta t}\rangle-2\langle u_t,(A_u(u_\theta,u_\theta))_t\rangle\\
&=&\partial_\theta^2|u_t|^2-2|u_{t\theta}|^2
+2\langle u_{tt},A_u(u_\theta,u_\theta)\rangle\\
&=&\partial_\theta^2|u_t|^2-2|u_{t\theta}|^2
+2\langle A_u(u_t,u_t),A_u(u_\theta,u_\theta)\rangle.
\end{eqnarray*}
Thus by (\ref{Fiftheqn}), if $t>t_3=t_1+(t_2-t_1)/2$,
\begin{equation}
\label{Twelftheqn}
(\partial_t-\partial_\theta^2)|u_t|^2-4(t_2-t_1)^
{-1}\sup_M|A|^2\cdot
E(u_0)\cdot|u_t|^2\leq 0.
\end{equation}
Hence
\begin{eqnarray*}
\label{Eightheqn}
\int_{S^1}|u_t|^2(t_2,\theta)d\theta&\leq&\inf_{t_3\leq t\leq t_2}\int_{S^1}|u_t|^2(t,\theta)d\theta+C(t_2-t_1)^{-1}\int_{t_1}^{t_2}\int_{S^1}
|u_t|^2d\theta dt\\
&\leq&(C+2)(t_2-t_1)^{-1}\int_{t_1}^{t_2}\int_{S^1}
|u_t|^2d\theta dt,
\end{eqnarray*}
where $C$ depends on $M$ and $E(u_0)$. Combining the inequality above, (\ref{Seventheqn}) and Lemma \ref{SecondLem}, there exists $\delta>0$ so that
 (\ref{Sixtheqn}) holds.
\end{proof}
It follows from Lemma \ref{FourthLem} and (\ref{Fiftheqn}) that there exists $R_0>0$, depending only on $M$ and $u_0$, so that for $t\geq 0$, $\sqrt{2\pi}\cdot\sup_M|A|^2\cdot\int_{\{t\}\times I_{R_0}}|u_\theta|^2d\theta<1/16$ , where $I_{R_0}$ is any segment on unit circle of length $R_0$. To prove the uniqueness of the solution of (\ref{Firsteqn}), we need the following lemma.
\begin{lem}
\label{SixthLem}
Suppose that $u$ is a solution of (\ref{Firsteqn}) in $C^\infty((0,\infty)\times S^1,M)$. Then
\begin{equation}
\label{Thirteentheqn}
\int^T_0\int_{S^1}|u_{\theta\theta}|^2d\theta dt\leq \frac{T}{4R_0^2}E(u_0)+2\left[E(u_0)-E(u(T,\cdot))\right].
\end{equation}
\end{lem}
\begin{proof}
If $0<t_0\leq T$, then by (\ref{Firsteqn}) and an interpolation inequality (see Lemma 6.7 in \cite{S2}),
\begin{eqnarray*}
\int^T_{t_0}\int_{S^1}|u_{\theta\theta}|^2d\theta dt&\leq &\int^T_{t_0}\int_{S^1}|u_t|^2d\theta dt+\sup_{M}|A|^2\cdot\int^T_{t_0}\int_{S^1}|u_\theta|^4d\theta dt\\
&\leq& \left[E(u_0)-E(u(T,\cdot))\right]+\frac{1}{2}\int^T_{t_0}\int_{S^1}
|u_{\theta\theta}|^2d\theta dt+\frac{T}{8R_0^2}E(u_0).
\end{eqnarray*}
Absorbing the righthand side into the lefthand side and noting that the estimate is independent of $t_0$, (\ref{Thirteentheqn}) follows immediately.
\end{proof}
Now we are ready to show the uniqueness of the solution of the harmonic map heat flow.
\begin{lem}
\label{FifthLem}
Given $u_0\in H^1(S^1,M)$, let $u$ and $\tilde{u}$ be solutions of (\ref{Firsteqn}) in $C^\infty((0,\infty)\times S^1,M)$.
Then $u=\tilde{u}$.
\end{lem}
\begin{proof}
Define $v=u-\tilde{u}$, thus 
\begin{equation}
\label{Ninetheqn}
v_t=v_{\theta\theta}
-A_{u}(u_\theta,u_\theta)
+A_{\tilde{u}}(\tilde{u}_\theta,\tilde{u}_\theta).
\end{equation}
Multiply both sides of $(\ref{Ninetheqn})$ by $v$ and integrate over $[0,t_0]\times S^1$,  
\begin{eqnarray*}
&&\int_{\{t_0\}\times S^1}|v|^2d\theta
+2\int_0^{t_0}\int_{S^1}|v_\theta|^2d\theta dt\\
&=&2\int_0^{t_0}\int_{S^1}\langle A_{\tilde{u}}(\tilde{u}_\theta,\tilde{u}_\theta)
-A_{u}(u_\theta,u_\theta),v\rangle d\theta dt\\
&\leq&C(M)\int_0^{t_0}\int_{S^1}|v|^2(|\tilde{u}_\theta|^2
+|u_\theta|^2)d\theta dt+
C(M)\int_0^{t_0}\int_{S^1}|v||v_\theta|(|\tilde{u}_\theta|
+|u_\theta|)d\theta dt\\
&\leq&C(M)\int_0^{t_0}\int_{S^1}|v|^2(|\tilde{u}_\theta|^2
+|u_\theta|^2)d\theta dt+
\int_0^{t_0}\int_{S^1}|v_\theta|^2 d\theta dt\\
&\leq&C(M)\int_0^{t_0}\left(\norm{u_\theta}_{C^0(\{t\}\times S^1)}^2+\norm{\tilde{u}_\theta}_{C^0(\{t\}\times S^1)}^2\right)\int_{S^1}|v|^2d\theta dt+
\int_0^{t_0}\int_{S^1}|v_\theta|^2 d\theta dt.\\
\end{eqnarray*}
By Lemma \ref{SecondLem}, \ref{SixthLem} and Sobolev embedding theorem, there exists $\delta>0$, depending on $M$ and $u_0$, so that if $t_0\leq \delta$, then
\begin{eqnarray*}
&&C(M)\int_0^{t_0}\norm{\tilde{u}_\theta}_{C^0(\{t\}\times S^1)}^2+\norm{u_\theta}_{C^0(\{t\}\times S^1)}^2dt\\
&\leq&C(M)\int_0^{t_0}\int_{S^1}|\tilde{u}_{\theta\theta}|^2
+|u_{\theta\theta}|^2d\theta dt\\
&\leq&C(M)\left[E(u_0)\frac{t_0}{2R_0^2}
+4E(u_0)-2E(u(t_0,\cdot))-2E(\tilde{u}(t_0,\cdot))\right]\leq\frac{1}{2}.
\end{eqnarray*} 
Thus, absorbing the righthand side into the lefthand side, 
\begin{equation}
\sup_{0\leq t\leq \delta}\int_{\{t\}\times S^1}|v|^2d\theta+2\int_0^\delta\int_{S^1}|v_\theta|^2 d\theta dt\leq 0.
\end{equation}
Since $[0,T]$ is compact, Lemma \ref{FifthLem} follows by iteration.
\end{proof}

\section{Width and Good Sweepouts}
In \cite{CM1}, Colding and Minicozzi introduced the crucial geometric concepts: sweepout and width. 
\begin{defn}
\label{SecondDefn} A continuous map $\sigma: [-1,1]\times S^1\To M$ is called a sweepout on $M$,  if $\sigma(s,\cdot)\in H^1(S^1,M)$ for each $s\in [-1,1]$, the map $s\To\sigma(s,\cdot)$ is continuous from $[-1,1]$ to $H^1(S^1,M)$ and $\sigma$ maps $\{-1\}\times S^1$ and $\{1\}\times S^1$ to points. 
\end{defn}
The sweepout $\sigma$ induces a map $\tilde{\sigma}$ from the sphere $S^2$ to $M$ and we will not distinguish $\sigma$ with $\tilde{\sigma}$. Denote by $\Omega$ the set of sweepouts on $M$. The homotopy class $\Omega_{\sigma}$ of $\sigma$ is the path connected component of $\sigma$ in $\Omega$, where the topology is induced from $C^0([-1,1],H^1(S^1,M))$.
\begin{defn}
\label{ThirdDefn} The width $W=W(\Omega_{\sigma})$ of the homotopy class $\Omega_{\sigma}$ is defined by 
\begin{equation}
\label{Fourtheqn}
W=\inf_{\hat{\sigma}\in\Omega_{\sigma}}
\max_{s\in[-1,1]}E(\hat{\sigma}(s,\cdot)).
\end{equation}
\end{defn}
For $\alpha\in (0,1)$ fixed, let $\gamma:S^1\To M$ be 
a smooth closed curve and $G$ be the set of closed geodesics 
in $M$. Define $\dist_\alpha(\gamma,G)=\inf_{\tilde{\gamma}\in G}
\norm{\gamma-\tilde{\gamma}}_{C^{1,\alpha}(S^1)}$.
We prove the following proposition for the solution of 
(\ref{Firsteqn}), which is the key to the proof of Theorem \ref{SecondMainThm}:
\begin{prop}
\label{FirstProp} Given $0<\alpha<1$, $W_0\geq 0$, $t_0>0$ and $\epsilon>0$, there exists $\delta_0>0$ so that if $W_0-\delta_0\leq E(u(t_0,\cdot))\leq E(u_0)\leq W_0+\delta_0$, then $\dist_\alpha(u(t_0,\cdot),G)<\epsilon$.
\end{prop}
\begin{proof}
If Proposition \ref{FirstProp} fails, then there exist $0<\alpha<1$, $W_0\geq 0$, $t_0>0$, $\epsilon>0$, and a sequence of solutions $u^j$ of the harmonic map heat flow satisfying that $W_0-1/j\leq E(u^j(t_0,\cdot))\leq E(u^j_0)\leq W_0+1/j$ and $\dist_\alpha(u^j(t_0,\cdot),G)\geq\epsilon$. It follows from the evolution equation of $|u^j_t|^2$ (see (\ref{Twelftheqn})), Lemma \ref{SecondLem} and Theorem 7.36 in \cite{Lie} that 
\begin{equation}
\sup_{\theta\in S^1}|u^j_t|^2(t_0,\theta)\leq 
C\left[E(u^j(t_0/2,\cdot))-E(u^j(t_0,\cdot))\right],
\end{equation}
where $C$ depends on $M$, $t_0$ and $W_0$.
Thus $\sup_{\theta\in S^1}|u^j_t|(t_0,\theta)\To 0$ and it follows from (\ref{Fiftheqn}) that $\norm{u^j(t_0,\cdot)}_{C^2(S^1)}$ is uniformly bounded by constants depending on $M$, $t_0$ and $W_0$. Therefore by Arzela-Ascoli's Theorem and Theorem 1.5.1 in \cite{He}, there exists a subsequence (relabelled) $u^j(t_0,\cdot)$ converges to $u^\infty$ in $C^{1,\alpha}(S^1,M)$ and $u^\infty$ is a closed geodesic in $M$. This is a contradiction. 
\end{proof} 
Let $\sigma$ be a sweepout on closed manifold $M$ and $\sigma^j$ be a minimizing sequence of sweepouts in $\Omega_\sigma$, that is 
\begin{equation}
W\leq \max_{s\in [-1,1]}E(\sigma^j(s,\cdot))\leq W+1/j.
\end{equation}
Applying the harmonic map heat flow to each slice of $\sigma^j$, we get a map $\Phi^j:[-1,1]\times [0,\infty)\times S^1\To M$ and for each $s\in [-1,1]$ fixed, $\Phi^j(s,t,\theta)$ solves $(\ref{Firsteqn})$ with $\Phi^j(s,0,\theta)=\sigma^j(s,\theta)$. It follows from the 
proof of the long time existence and uniqueness of the solution of (\ref{Firsteqn}) that for any $t_0\geq 0$, the map $s\To \Phi(s,t_0,\cdot)$ is continuous from $[-1,1]$ to $H^1(S^1,M)$ and thus $\Phi^j(\cdot,t_0,\cdot)$ is still a sweepout on $M$.
Since $[-1,1]$ is compact, for any $\epsilon>0$, there exists $\delta>0$ so that: if $0\leq t_1<t_2\leq t_0$ and $t_2-t_1<\delta$, then $\int_{t_1}^{t_2}\int_{S^1}|\Phi_t^j(s,t,\theta)|^2d\theta dt<\epsilon$ for any $s\in [-1,1]$.
Hence by Lemma \ref{FourthLem}, for any $t_0>0$, $\Phi^j(\cdot,t_0,\cdot)$ is homotopic to $\sigma^j$. Therefore it follows from Proposition \ref{FirstProp} that the $\Phi^j(\cdot,t_0,\cdot)$ are good sweepouts on $M$. That is,
\begin{thm}
\label{SecondMainThm} Given $0<\alpha<1$, $t_0>0$ and $\epsilon>0$, there exists $\delta>0$ so that if $j>1/\delta$ and $s\in [-1,1]$ satisfies $E(\Phi^j(s,t_0,\cdot))\geq W-\delta$ \footnote[2]{Such $s$ exists, since $W\leq\max_{s\in [-1,1]}E(\Phi^j(s,t_0,\cdot))\leq W+1/j$.}, then $\dist_\alpha(\Phi^j(s,t_0,\cdot),G)<\epsilon$.
\end{thm}
In \cite{CM1}, Colding and Minicozzi show that the width is non-negative and is positive if $\sigma$ is not the zero element in $\pi_2(M)$. In fact, assume that $W(\Omega_\sigma)=0$ and $\hat{\sigma}\in\Omega_\sigma$ is such that the energy of each slice of $\hat{\sigma}$ is sufficiently small. Then each slice, $\hat{\sigma}(s,\cdot)$, is contained in a strictly convex neighborhood of $\hat{\sigma}(s,\theta_0)$ and note that $s\To\hat{\sigma}(s,\theta_0)$ is a continuous curve on $M$. Hence a geodesic homotopy connects $\hat{\sigma}$ to a path of point curves and thus $\hat{\sigma}$ is homotopically trivial. Since $G$ is closed in the $H^1(S^1,M)$ topology, we have
\begin{cor}
\label{FirstCor}
If $M$ is a closed Riemannian manifold and $\pi_2(M)\neq\{0\}$, then there exists at least one non-trivial closed geodesic on $M$.
\end{cor}
\begin{rem}
\label{SecondRem}
Instead of using the unit interval $[-1,1]$ as the parameter space for the curves in the sweepout and assuming that the curves start and end in point curves, we could have used any compact space $\mathcal{P}$ and required that the curves are constants on $\partial\mathcal{P}$. In this case, $\Omega^{\mathcal{P}}$ is the set of continuous maps $\sigma: \mathcal{P}\times S^1\To M$ so that for
each $s\in\mathcal{P}$ the curve $\sigma(s,\cdot)$ is in $H^1(S^1,M)$, the map $s\To\sigma(s,\cdot)$ is continuous from $\mathcal{P}$ to $H^1(S^1,M)$, and
finally $\sigma$ maps $\partial\mathcal{P}$ to point curves. Given $\sigma\in\Omega^{\mathcal{P}}$, the homotopy class $\Omega^{\mathcal{P}}_\sigma$ is the set of maps $\hat{\sigma}\in\Omega^{\mathcal{P}}$ that are homotopic to $\sigma$ through maps in $\Omega^{\mathcal{P}}$. And the width $W=W(\Omega^{\mathcal{P}}_\sigma)$ is defined by 
\begin{equation}
W=\inf_{\hat{\sigma}\in\Omega^{\mathcal{P}}_\sigma}
\max_{s\in\mathcal{P}}E(\hat{\sigma}(s,\cdot)).
\end{equation}
Theorem \ref{SecondMainThm} holds for general parameter space; the proof is virtually the same when $\mathcal{P}=[-1,1]$.
\end{rem}

\end{document}